\documentclass[a4paper,11pt]{amsart}
\usepackage{graphicx} 
\usepackage[a4paper, margin=3 cm]{geometry}
\usepackage{hyperref}
\usepackage{amsthm}
\usepackage{thm-restate}
\usepackage{geometry}
\usepackage{amssymb}
\usepackage{amsmath}
\usepackage{cleveref}
\usepackage{comment}
\usepackage{setspace}
\usepackage{xcolor}
\usepackage[shortlabels]{enumitem}
\usepackage{listings}
\usepackage{blkarray}
\usepackage{lmodern}
\usepackage[english]{babel}
\usepackage{float}
\usepackage{tikz}
\usetikzlibrary{arrows.meta,arrows}
\usetikzlibrary{arrows,decorations.markings}
\usetikzlibrary{decorations.pathmorphing}

\newtheorem{theorem}{Theorem}[section]

\newtheorem{claim}[theorem]{Claim}
\newtheorem*{claim*}{Claim}

\newtheorem{conjecture}[theorem]{Conjecture}
\newtheorem{remark}[theorem]{Remark}

\theoremstyle{definition}

\title{Triangle-free $d$-degenerate graphs have small fractional chromatic number}
\author{Anders Martinsson}

\thanks{Institute of Theoretical Computer Science, Department of Computer Science, ETH Z\"{u}rich, Switzerland.}
\thanks{\texttt{anders.martinsson@inf.ethz.ch}}

\date{\today}

\begin{document}
\begin{abstract} A well-known conjecture by Harris states that any triangle-free $d$-degenerate graph has fractional chromatic number at most $O\left(\frac{d}{\ln d}\right)$. This conjecture has gained much attention in recent years, and is known to have many interesting implications, including a conjecture by Esperet, Kang and Thomass\'e that any triangle-free graph with minimum degree $d$ contains a bipartite induced subgraph of minimum degree $\Omega(\log d)$. Despite this attention, Harris' conjecture has remained wide open with no known improvement on the trivial upper bound, until now.

In this article, we give an elegant proof of Harris' conjecture. In particular, we show that any triangle-free $d$-degenerate graph has fractional chromatic number at most $(4+o(1))\frac{d}{\ln d}.$ The conjecture of Esperet et al. follows as a direct consequence. We also prove a more general result, showing that for any triangle-free graph $G$, there exists a random independent set in which each vertex 
$v$ is included with probability $\Omega(p(v))$, where $p:V(G)\rightarrow [0,1]$ is any function that satisfies a natural condition.
\end{abstract}

\maketitle

\section{Introduction}

The \emph{fractional chromatic number} $\chi_f(G)$ of a graph $G$ is a linear relaxation of the chromatic number $\chi(G)$ defined as the smallest real number $k\geq 1$ such that there exists a probability distribution over the independent sets of $G$ such that each vertex is present with probability at least $1/k$. Let $\Delta(G)$ denote the maximum degree in $G$ and let $d(G)$ denote the degeneracy of $G$. It is well-known through greedy coloring that
$$ \chi_f(G) \leq \chi(G)\leq d(G)+1 \leq \Delta(G)+1,$$
where, by the classical Brooks' Theorem \cite{Brooks41}, equality holds throughout if and only if some connected component of $G$ is a complete graph on $\Delta(G)+1$ vertices.

It is natural to ask if these upper bounds can be improved under the assumption that $G$ is somehow ``far'' from being a complete graph. A prime example of this is the seminal result of Johansson \cite{Johansson96} that $\chi(G)\leq O\left(\frac{\Delta(G)}{\log \Delta(G)}\right)$ for any \emph{triangle-free} graph. A more recent breakthrough result by Molloy \cite{Molloy19} improves this to
$$\chi(G) \leq (1+o(1))\frac{\Delta(G)}{\ln(\Delta(G))}.$$
This is known to be best possible up to at most a factor $2$, see \cite{Bollobas81}. These results have paved the way for a long list of interesting strengthenings and generalizations in the literature, see for instance \cite{Alon99,Anderson23,Anderson24,Bonamy22,Bradshaw23,Davies20B,Hurley21A,Hurley21B,Hurley23,Martinsson21}.

Observe that Molloy's result implies $\chi_f(G)\leq (1+o(1))\frac{\Delta(G)}{\ln \Delta(G)}$ for all triangle-free graphs, which is also the best known upper bound for the fractional chromatic number in terms of the maximum degree. Nevertheless, the literature on the fractional chromatic number, or, more generally, random processes over the independent sets of triangle-free graphs is quite rich. A natural example of this is the uniform distribution over the independent sets of a graph, which in turn generalizes to the so-called \emph{hard-core model}, see for instance  \cite{Davies18,Davies25}. For some recent examples of other processes over the independent sets of triangle-free graphs, see for instance \cite{Davies20A,Kelly24,Martinsson24,Pirot21}.

It is tempting to think that the upper bound on $\chi(G)$ in terms of degeneracy can be improved under the assumption of triangle-freeness -- similarly to Johansson's and Molloy's results. However, this turns out to be false. A number of articles ranging back to at least the 1940s, see for instance \cite{Alon99,Descartes54,Kostochka99,Zykov49},
present constructions of $d$-degenerate triangle-free graphs with chromatic number $d+1$. Instead, a well-known conjecture by Harris \cite{Harris19} proposes that the fractional chromatic number of a triangle-free graph should have this relationship to degeneracy.
\begin{conjecture}[cf. Conjecture 6.2 in \cite{Harris19}]\label{conj:harris} Suppose that $G$ is $d$-degenerate and triangle-free. Then $\chi_f(G)= O(d/\log d)$.
\end{conjecture}

This conjecture has gained quite some attention in recent years. It is known to imply various other conjectures and strengthenings of known results in the literature \cite{Esperet19,Harris19,Janzer24,Kwan20,Martinsson24} including another well-known conjecture by Esperet, Kang and Thomass\'e \cite[Conjecture~1.5]{Esperet19} that any triangle-free graph with minimum degree $d$ contains an induced bipartite subgraph of minimum degree $\Omega(\log d).$ Currently the best known lower bound on the conjecture by Esperet et al. is $\Omega(\log d/\log\log d)$ due to Kwan, Letzter, Sudakov and Tran \cite{Kwan20}, though Gir\~{a}o and Hunter (personal communication) recently announced upcoming work improving this to average degree $(1-o(1))\ln d$. See also \cite{Davies20B,Kelly24}. Despite this attention, Harris' conjecture itself has remained wide open with no asymptotic improvement on the trivial upper bound of $O(d)$ having existed in the literature -- until now.

In this paper, we give an elegant proof of Conjecture \ref{conj:harris}. More precisely, our main result states as follows.
\begin{theorem}\label{thm:main} Suppose $G$ is a triangle-free and $d$-degenerate graph. Then $\chi_f(G)\leq (4+o(1))\frac{d}{\ln d},$ where the $o(1)$ term tends to $0$ as $d$ increases.
\end{theorem}

As already mentioned, this result is known to have some nice consequences. For instance, a direct application of the theorem gives that any triangle-free graph with minimum degree $d$ contains an induced bipartite subgraph of average degree at least $(\frac14-o(1))\ln d$ (and thus one of minimum degree at least $(\frac18-o(1))\ln d$). We refer to \cite[Theorem~3.1]{Esperet19} for further details on the calculations. This proves \cite[Conjecture~1.5]{Esperet19}, improving on the previously best known bound \cite{Kwan20} by a factor $\Theta(\log \log d)$. We note that the factor $\frac14$ can likely be improved by a more careful analysis, but we do not attempt this here.

In addition, Harris \cite{Harris19} observed that Theorem \ref{thm:main} can be extended to the setting where the triangle-free condition is relaxed to $G$ being \emph{locally sparse}, similar to the extension of the upper bound for the chromatic number of triangle-free graphs presented in \cite{Alon99}. More precisely, we say that a $d$-degenerate graph $G$ has \emph{local triangle bound} $y$ if each vertex in $G$ is the last vertex of at most $y$ triangles, where \emph{last} refers to the degeneracy ordering of the graph. Combining Theorem \ref{thm:main} with \cite[Lemma~6.3]{Harris19}, it follows that
$$ \chi_f(G)= O\left(\frac{d}{\ln(d^2/y)}\right) $$
for any $d$-degenerate graph $G$ with local triangle bound $y$. This in turn proves various relationships between the chromatic number and the triangle count of a graph. We refer to \cite[Section~6]{Harris19} for more details.

In fact, at the cost of slightly weakening the constant $4+o(1)$, Theorem \ref{thm:main} can be seen as a special case of the following generalization of Harris' conjecture.

\begin{theorem}\label{thm:maingen} Let $G$ be a triangle-free graph with a vertex ordering $v_1, v_2,\dots , v_n$. Suppose $p:V(G)\rightarrow [0, 1]$ satisfies
$$ p(v_i) \leq \prod_{v_j \in N_L(v_i)} \left(1-p(v_j)\right)$$
for all vertices $v_i$, where $N_L(v_i)$ denotes the set of neighbors $v_j$ of $v_i$ with $j<i$. Then there exists a probability distribution $\mathcal{I}$ over the independent sets of $G$ such that
$$\mathbb{P}_{I\sim \mathcal{I}}(v_i\in I)\geq \alpha \cdot p(v_i),  $$ for all vertices $v_i$, where $\alpha:=\frac{1-e^{-\frac12}}{2}=.196\dots$
\end{theorem}

It is not too hard to see that this statement implies Harris' conjecture. 
Moreover, by ordering the vertices of any triangle-free graph decreasingly by their degrees it follows that the conditions of Theorem \ref{thm:maingen} are satisfied for $p(v_i)=\Theta\left( \frac{\ln d(v_i)}{d(v_i)}\right)$, where $d(v_i)$ denotes the degree of $v_i$. This recovers, up to constant factors, the so-called \emph{local Shearer bound} as conjectured by Kelly and Postle \cite[Conjecture~2.2]{Kelly24} and recently proven by the author and Steiner \cite[Theorem~1.2]{Martinsson24}. Beyond this, Theorem \ref{thm:maingen} appears to be a very natural extension of Harris' conjecture which may be of independent interest.

Finally, given the resolution of Harris' conjecture, a natural remaining question is to determine the optimal leading constant $C$ for the problem. In particular, by combining Theorem \ref{thm:main} with \cite{Bollobas81}, we know that $\frac12\leq C \leq 4.$ It would appear that the most reasonable answer is $C=1$. We state this as a conjecture.
\begin{conjecture} The following holds for any sufficiently large $d$.
\begin{enumerate}[(i)]
    \item $\chi_f(G)\leq (1+o(1))\frac{d}{\ln d}$ for all $d$-degenerate triangle-free graphs $G$.
    \item There exists a $d$-degenerate triangle-free graph $G$ such that $\chi_f(G)\geq (1-o(1))\frac{d}{\ln d}.$
\end{enumerate}
\end{conjecture}

\section{Proof of Theorems \ref{thm:main} and \ref{thm:maingen}}
Let $G$ be a triangle-free graph with vertices $v_1, \dots, v_n$. Let $N_L(v_i)$ denote the set of neighbors $v_j$ of $v_i$ where $j<i$ and let $N_R(v_i)$ denote the set of neighbors $v_j$ of $v_i$ where $j>i$. Consider the following process:

Let $w_0:V(G)\rightarrow \mathbb{R}_{> 0}$ be an assignment of positive weights to the vertices of $G$. Initially assign vertices the weights $w(v_i)=w_0(v_i)$ for all $1\leq i\leq n$. Then for each step $i$ in $1, 2, \dots, n$ do the following.
\begin{itemize}
    \item With probability $1-e^{-w(v_i)}$, put $w(v_j)=0$ for all $v_j\in N_R(v_i)$.
    \item With probability $e^{-w(v_i)}$, multiply the weight of all $v_j\in N_R(v_i)$ by $e^{w(v_i)}$.
\end{itemize}

Let $I$ be the set of vertices $v_i$ for which the first option occurred. It is easy to see that $I$ is an independent set. If $v_i\in I$, then at step $i$ all vertices $v_j\in N_R(v_i)$ get assigned the weight $0$ for the rest of the process, which means they enter the independent set with probability $1-e^{-0}=0.$

Theorems \ref{thm:main} and \ref{thm:maingen} are both direct consequences of the following result.
\begin{theorem}\label{thm:mainproc} For any weight function $w_0:V(G)\rightarrow \mathbb{R}_{> 0}$ the following holds. For any vertex $v_k\in V(G)$ and any $\varepsilon>0$, if $$w_0(v_k) \exp\left(2\sum_{v_i\in N_L(v_k)} w_0(v_i)\right)\leq \varepsilon,$$ then
$$\mathbb{P}(v_k\in I) \geq \frac{1-e^{-\varepsilon}}{2\varepsilon}w_0(v_k).$$
\end{theorem}

Fix a vertex $v_k$. For the remainder of this section, we will work to compute a lower bound on the probability that $v_k\in I$. Thus proving Theorem \ref{thm:mainproc}.

In order to analyze this, let us consider a modified process. Initially assign the vertices weights $\tilde{w}(v_i)=w_0(v_i)$ for all $1\leq i\leq n$. Then for each step $i$ in $1, 2, \dots, k-1$, do the following.
\begin{itemize}
    \item If $v_i\not\in N_L(v_k)$, do the same update rule as for $w$.
    \item If $v_i\in N_L(v_k)$, multiply the weight of all vertices $v_j\in N_R(v_i)$ by $e^{\tilde{w}(v_i)}$.
\end{itemize}
In other words, $\tilde{w}$ has the same update rule as $w$ for any step $i$ where $v_i\not\in N_L(v_k)$. For any step $i$ where $v_i\in N_L(v_k)$, the process follows the update rule of the second bullet point of $w$ with probability $1$.

Let us denote by $w_i(v_j)$ and $\tilde{w}_i(v_j)$ the weight of $v_j$ after step $i$ in the respective processes, let $\tilde{w}_0(v_j):=w_0(v_j)$, and let $$X:=\sum_{v_i\in N_L(v_k)} \tilde{w}_{k-1}(v_i).$$ By construction of $\tilde{w}$, we have the following relation to $w$.
\begin{claim}\label{claim:procrel} For any function $f$ such that $f(0)=0$ we have
$$\mathbb{E}\left[f(w_{k-1}(v_k))\right]=\mathbb{E}\left[f(\tilde{w}_{k-1}(v_k)) e^{-X} \right]. $$
\end{claim}
\begin{proof} We can encode each possible sequence of weight functions $(w_0, w_1, \dots, w_{k-1})$ of the processes $w$ as a sequence $a\in \{1,2\}^{k-1}$ where $a_i$ denotes whether, in step $i$, randomness chooses the first or the second bullet point. In other words, $a_i=1$ if and only if $v_i\in I$.

Note that if $a_i=1$ for any index $i$ where $v_i\in N_L(v_k)$ then this sequence will result in $w_{k-1}(v_k)=0$. Thus, such a sequence does not contribute to the value of $\mathbb{E}\left[f(w_{k-1}(v_k))\right]$. Similarly, if $a_i=a_j=1$ for any two neighboring vertices $v_i$ and $v_j$ then, the probability of the corresponding sequence is $0$, which means it also does not contribute to $\mathbb{E}\left[f(w_{k-1}(v_k))\right]$.

Let $A\subseteq \{1,2\}^{k-1}$ denote the set of sequences that do not match either of the aforementioned conditions. Then any $a\in A$ can be interpreted as a possible sequence of weight functions $(w_0, \dots, w_{k-1})$ and $(\tilde{w}_0, \dots, \tilde{w}_{k-1})$ produced by either process $w$ or $\tilde{w}$. Note that, by definition of $w$ and $\tilde{w}$, the same sequence of choices $a$ will produce the same sequence of weight functions in either process. Let us denote this common sequence by $w^a$, and let us denote by $\mathbb{P}_w(a)$ and $\mathbb{P}_{\tilde{w}}(a)$ the probabilities that the sequence of choices of the respective processes equals $a$.

By comparing the transition probabilities of $w$ and $\tilde{w}$, we immediately get
\begin{align*}
    \frac{\mathbb{P}_w(a)}{\mathbb{P}_{\tilde{w}}(a)} &= \exp\left(-\sum_{v_i\in N_L(v_k)} w^a_{i-1}(v_i)\right) =\exp\left(-\sum_{v_i\in N_L(v_k)} w^a_{k-1}(v_i)\right)
\end{align*}
for all $a\in A$, where the last equality follows by observing that no vertex $v_i$ has its weight updated after step $i-1$. Thus
\begin{align*}
    \mathbb{E}[f(w_{k-1}(v_k)]
    &=\sum_{a\in A} f(w^a_{k-1}(v_k)) \mathbb{P}_w(a)\\ 
    &=\sum_{a\in A} f(w^a_{k-1}(v_k))\exp\left(-\sum_{v_i\in N_L(v_k)} w^a_{k-1}(v_i)\right)\mathbb{P}_{\tilde{w}}(a)\\
    &=\mathbb{E}[f(\tilde{w}_{k-1}(v_k))e^{-X} ].
\end{align*}
\end{proof}

\begin{claim}\label{claim:martingale} Suppose $v_i\in N_L(v_k)$. Then $\tilde{w}_t(v_i)$ is a martingale in $t$ for $t=0, \dots, k-1$.
\end{claim}
\begin{proof} By definition of $\tilde{w}$, the only steps $j$ where the value of $\tilde{w}(v_i)$ is updated are those where $v_j\in N_L(v_i)$. Note that $v_j\not\in N_L(v_k)$ as otherwise $v_i, v_j, v_k$ would form a triangle. Thus $\tilde{w}(v_i)$ is updated according to
$$ \tilde{w}_j(v_i) = \begin{cases} 0&\text{ with probability }1-e^{-\tilde{w}_{j-1}(v_j)}\\ \tilde{w}_{j-1}(v_i)e^{\tilde{w}_{j-1}(v_j)}&\text{ with probability }e^{-\tilde{w}_{j-1}(v_j)}.\end{cases}$$
It is easy to see that this is preserved in expectation.
\end{proof}

\begin{claim}\label{claim:EX} $$\mathbb{E}X = \sum_{v_i \in N_L(v_k)} w_0(v_i).$$
\end{claim}
\begin{proof} By Claim \ref{claim:martingale}, $\mathbb{E}X=\sum_{v_i\in N_L(v_k)}\mathbb{E}\tilde{w}_{k-1}(v_i) = \sum_{v_i\in N_L(v_k)}\mathbb{E}\tilde{w}_{0}(v_i). $
\end{proof}

\begin{claim}\label{claim:finalweight} $$\tilde{w}_{k-1}(v_k)=w_0(v_k) e^X.$$
\end{claim}
\begin{proof} By definition of $\tilde{w}$, $\tilde{w}(v_k)$ increases by a factor $e^{\tilde{w}_{i-1}(v_i)}=e^{\tilde{w}_{k-1}(v_i)}$ for each step $i$ where $v_i\in N_L(v_k)$. For any other step, $\tilde{w}(v_k)$ is unchanged.
\end{proof}

\begin{claim}\label{claim:vkinI} $$ \mathbb{P}(v_k\in I) = \mathbb{E}\left[ \left(1-e^{-w_0(v_k) e^X}\right)e^{-X} \right].$$
\end{claim} 
\begin{proof} By the definition of $w$ and $I$ we have $\mathbb{P}(v_k\in I) = \mathbb{E}[1-e^{-w_{k-1}(v_k)}].$ Let $f(x)=1-e^{-x}$. By Claim \ref{claim:procrel}, noting that $f(0)=0$, we get
$$ \mathbb{E}[1-e^{-w_{k-1}(v_k)}] = \mathbb{E}[f(w_{k-1}(v_k))] = \mathbb{E}[f(\tilde{w}_{k-1}(v_k))e^{-X}  ]. $$ By Claim \ref{claim:finalweight}, $\tilde{w}_{k-1}(v_k)=w_0(v_k) e^X$. Combining these gives the desired equality.
\end{proof}

\begin{proof}[Proof of Theorem \ref{thm:mainproc}] By Claim \ref{claim:EX} and Markov's inequality we have $$\mathbb{P}\left(X\geq  2\sum_{v_i \in N_L(v_k)} w_0(v_k)\right)\leq \frac{1}{2}.$$ By assumption of the theorem, we have $ w_0(v_k) e^X \leq \varepsilon $ whenever $X< 2\sum_{v_i \in N_L(v_k)} w_0(v_k)$. Using the elementary inequality $1-e^y\geq \frac{1-e^{-\varepsilon}}{\varepsilon} y$ valid for any $0\leq y\leq \varepsilon $ it follows that 
$$\left(1-e^{-w_0(v_k) e^X }\right)e^{-X} \geq \frac{1-e^{-\varepsilon}}{\varepsilon} w_0(v_k) e^X\cdot e^{-X}=\frac{1-e^{-\varepsilon}}{\varepsilon} w_0(v_k)$$ for any such $X$. Combining this with Claim \ref{claim:vkinI}, we immediately get
$$\mathbb{P}(v_k \in I)=\mathbb{E}\left[\left(1-e^{-w_0(v_k) e^X}\right)e^{-X}\right]  \geq \frac12 \frac{1-e^{-\varepsilon}}{\varepsilon} w_0(v_k).$$
\end{proof}

\begin{proof}[Proof of Theorem \ref{thm:main}.] Let $G$ be a triangle-free $d$-degenerate graph with degeneracy order $v_1, \dots, v_n$ such that $|N_L(v_i)|\leq d$ for all vertices $v_i$. Assume $d\geq 3.$ We apply Theorem \ref{thm:mainproc} with $w_0\equiv \frac{\ln d -2\ln\ln d}{2d}.$ One immediately checks that
$$ w_0(v_k) \exp\left(2\sum_{v_i\in N_L(v_k)} w_0(v_i) \right) \leq \frac{1}{2\ln d}=:\varepsilon,$$
which implies that
$$\mathbb{P}(v_k\in I) \geq \frac{1-e^{-\varepsilon}}{2\varepsilon} w_0(v_k)=\left(\frac14-o(1)\right)\frac{\ln d}{d},$$
where the last equality uses that $\frac{1-e^{-\varepsilon}}{\varepsilon}\rightarrow 1$ as $\varepsilon\rightarrow 0$.
\end{proof}
\begin{proof}[Proof of Theorem \ref{thm:maingen}] We apply Theorem \ref{thm:mainproc} with $w_0(v_i):=\frac12 p(v_i)$. Then
$$ w_0(v_k) \exp\left(2\sum_{v_i\in N_L(v_k)} w_0(v_i) \right) \leq \frac12 p(v_k) \prod_{v_i\in N_L(v_k)}\left(1-p(v_i)\right)^{-1} \leq \frac12=:\varepsilon,$$
which implies that
$$\mathbb{P}(v_k\in I) \geq \frac{1-e^{-\varepsilon}}{2\varepsilon} \frac12 p(v_k) = \frac{1-e^{-\frac12}}{2} p(v_k).$$
\end{proof}

\begin{remark} We note that it is possible to improve the constant $\alpha$ in Theorem \ref{thm:maingen} to $\frac14$ by a more elaborate analysis of the expectation in Claim \ref{claim:vkinI}, but, for the sake or brevity, we will not elaborate on this here.
\end{remark}

\section*{Acknowledgements} My thanks to Raphael Steiner for helping me to verify the proof, and to him, Antonio Gir\~{a}o, Zach Hunter and Charlotte Knierim for helpful discussions and feedback on earlier drafts.

\end{document}